\documentclass[10pt,reqno]{amsart}
\usepackage{amsmath,amscd,amsfonts,amsthm,amssymb,latexsym,mathrsfs}
\usepackage{tikz}
\usetikzlibrary{matrix,arrows,patterns}
\usepackage{tikz}
\usepackage{shuffle}

\theoremstyle{plain}
   \newtheorem{theorem}{Theorem}[section]
   \newtheorem{proposition}[theorem]{Proposition}
   \newtheorem{lemma}[theorem]{Lemma}
   \newtheorem{corollary}[theorem]{Corollary}

\theoremstyle{definition}
   
   \newtheorem{example}{Example}[section]

\theoremstyle{remark}
   \newtheorem{remark}[theorem]{Remark}

\newcommand{\NN}{\mathbb{N}}
\newcommand{\zz}{\mathbf{z}}
\newcommand{\ww}{\mathbf{w}}
\newcommand{\xx}{\mathbf{x}}
\newcommand{\yy}{\mathbf{y}}

\newcommand{\JJ}{\mathscr{L}}

\newcommand{\FQ}{\text{FQSym}}

\newcommand{\BB}{\text{B}}

\newcommand{\EE}{\mathcal{E}}
\newcommand{\FF}{\mathcal{F}}
\newcommand{\GG}{\mathcal{G}}

\newcommand{\RR}{\mathbb{R}}
\newcommand{\CC}{\mathbb{C}}

\newcommand{\sym}{\mathfrak{S}}

\renewcommand{\mod}{\mathop{\rm \ mod}}
\renewcommand{\Im}{{\rm Im}}
\renewcommand{\Re}{{\rm Re}}
\newcommand{\FQS}{\text{FQSym}}

\def\newop#1{\expandafter\def\csname #1\endcsname{\mathop{\rm
#1}\nolimits}}

\newop{diag}
\newop{AB}
\newop{DB}
\newop{B}
\newop{AT}
\newop{DT}
\newop{peak}
\newop{rank}
\newop{des}
\newop{Sym}
\newop{Int}
\newop{Orb}
\newop{Re}

\title{Multivariate $P$-Eulerian polynomials} 
\author{Petter Br\"and\'en}
\author{Madeleine Leander}
\address{Department of Mathematics, Royal Institute of Technology, SE-100 44 Stockholm,
Sweden}
\email{pbranden@kth.se}
\thanks{The first author is a Wallenberg Academy Fellow 
  supported by a grant from the Knut and Alice Wallenberg
  Foundation, and the Swedish Research Council (VR)}
\address{Department of Mathematics, Stockholm University, SE-106 91
  Stockholm, Sweden}
\email{madde@math.su.se}

\begin{document}
\begin{abstract}
The $P$-Eulerian polynomial  counts  the linear extensions of a labeled partially ordered set, $P$, by their number of descents. It is known that the 
$P$-Eulerian polynomials are real-rooted for various classes of posets $P$. The purpose of this paper is to extend these results to polynomials in several variables. To this end we study multivariate extensions of $P$-Eulerian polynomials and prove that for certain posets these polynomials are stable, i.e., non-vanishing whenever all variables are in the upper half-plane of the complex plane. 
A natural setting for our proofs is the Malvenuto-Reutenauer algebra of permutations (or the algebra of free quasi-symmetric functions). In the process we identify an algebra on Dyck paths, which to our knowledge has not been studied before.

\end{abstract}
\maketitle

\thispagestyle{empty}
\section{Introduction}
The Eulerian polynomials have been studied frequently in enumerative combinatorics, as well as in other areas since they first appeared in Euler's work \cite{Euler}, see \cite{Kyle}. 
The $n$th \emph{Eulerian polynomial} may be defined as the generating polynomial of the \emph{descent statistic} over the symmetric group $\sym_n$: 
$$
A_n(x) := \sum_{\pi \in \sym_n} x^{\des(\pi)+1},
$$
where $\des(\pi) := |\{ 1 \leq i \leq n-1 : \pi_i > \pi_{i+1}\}|$. An important property of the Eulerian polynomials is that all their zeros are real, i.e., $A_n(x)$ splits over $\RR$. This was already noted by Frobenius \cite{Fro}, and is not an isolated phenomenon as surprisingly many polynomials appearing in combinatorics are real-rooted, see \cite{FSurvey,PSurvey,RSurvey}. 

Recently a theory of multivariate stable (``real-rooted'') polynomials has been developed \cite{LYPSI,LYPSII,Wa3}. A multivariate polynomial is \emph{stable} if it is nonzero whenever all the variables have positive imaginary parts. Hence a univariate polynomial with real coefficients is stable if and only if all its zeros are real.  Efforts have been made to lift results concerning the zero distribution of univariate polynomials in combinatorics to concern multivariate extensions of the polynomials, see \cite{BrU,Mon,HV1,HV2}. There are several benefits of such a refinement. Firstly the stability of the multivariate polynomial implies the real-rootedness of the univariate polynomial. Secondly, the proofs of the multivariate statements are often simpler, and may lead to a better understanding of the combinatorial setting in question. Most importantly multivariate stability implies several inequalities, refining unimodality and log-concavity, among the coefficients, see \cite{BBL,Wa3}. 

An extension of the Eulerian polynomials to labeled posets was introduced in Stanley's thesis \cite{Stan}, and further studied in \cite{Brenti,NSC1,NSC2,RW,Wa1,Wa2}. We define a \emph{labeled poset} on $n$ elements to be a poset $P=([n],\preceq)$ where $\preceq$ is the partial order and  $\leq$ is the natural order on $[n]:=\{1,2,\ldots, n\}$. The \emph{Jordan-H\"older set} of $P$ is the set of linear extensions of $P$:  
$$
\mathscr{L}(P):=\{\pi \in \sym_n : \mbox{if } \pi_i \preceq \pi_j \mbox{ then, } i\leq j \mbox{ for all } i,j \in [n]\}, $$
where each permutation $\pi =\pi_1\pi_2 \cdots \pi_n$ in the symmetric group, $\sym_n$, is written in one-line notation. The \emph{$P$-Eulerian polynomial} is defined by 
\begin{equation}\label{PEul}
A_P(x) := \sum_{\pi \in \JJ(P)} x^{\des(\pi)+1}.
\end{equation}
Thus the Eulerian polynomial, $A_n(x)$, is the $P$-Eulerian polynomial of an $n$-element anti-chain, i.e., the poset on $n$ elements with no relations. 

The Neggers-Stanley conjecture asserted that for each labeled poset $P$, $A_P(x)$ is real-rooted, see \cite{FSurvey,PSurvey,Ne}. The conjecture was disproved in \cite{NSC3}, and for natural labelings it was disproved in \cite{Stem2}. The conjecture was proved for several classes of posets in \cite{Brenti,Wa1}, and it is still open for the important class of naturally labeled graded posets. 

In this paper we introduce and study a multivariate version of the $P$-Eulerian polynomials. We prove that these polynomials are stable for classes of labeled posets for which the univariate $P$-Eulerian polynomials are known to be real-rooted. In particular we prove that  stability of multivariate $P$-Eulerian polynomials respects disjoint unions of posets (Corollary \ref{label}). We argue that the natural context for this is the algebra of free quasi-symmetric functions \cite{ DHT, MR}. In the process we identify a graded algebra, $\mathcal{D}$, on Dyck paths, which to our knowledge has not been studied before. One of our main theorems may be formulated as: The multiplication in $\mathcal{D}$ preserves stability (for an appropriate notion of stability of weighted sums of Dyck paths). 

The multivariate $P$-Eulerian polynomial is also an extension of Stembridge's peak polynomial \cite{Stem1,Stem2}. We introduce a multivariate peak polynomial for labeled posets $P$, and prove that it is nonzero whenever all variables are in the open right half-plane of the complex plane, whenever the multivariate $P$-Eulerian polynomial is stable.

\section{Multivariate $P$-Eulerian polynomials}
\label{disjoint}
For a permutation $\pi=\pi_1\pi_2\cdots \pi_n \in \sym_n$,  let 
$$
\AB(\pi):=\{\pi_i \in [n] : \pi_i<\pi_{i+1}\}
\quad \mbox{ and } \quad \DB(\pi):=\{\pi_{i} \in [n] : \pi_{i-1}>\pi_{i}\},
$$
where $\pi_0=\pi_{n+1}:=\infty$,  
denote the set of \emph{ascent bottoms} and \emph{descent bottoms} of $\pi$, respectively. 

Let $[n]' := \{i' : i \in [n]\}$ be a distinct copy of $[n]$.  For a permutation $\pi \in \sym_n$ define a monomial in the variables $\zz=\{z_e: e\in [n]\cup [n]'\}$:
$$w_1(\pi):=\prod_{e\in \DB(\pi)}z_e \prod_{e\in \AB(\pi)}z_{e'}= \prod_{e\in \BB(\pi)}z_e,$$
where $\BB(\pi)= \DB(\pi) \cup \{e' : e \in \AB(\pi)\}$. 
The \emph{multivariate $P$-Eulerian polynomial} is defined as 
$$A_P(\zz) :=\sum_{\pi\in \mathscr{L}(P)}w_1(\pi).$$ 
\begin{example}
Let $P$ be the poset below. 
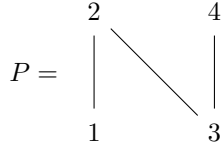
\begin{center}
 \begin{tikzpicture}
  [scale=.4,auto=left,every node/.style={circle}]
\node (n5) at (-2,2) {$P=$};
  \node (n1) at (0,0) {1};
  \node (n2) at (0,4)  {2};
  \node (n3) at (4,0)  {3};
  \node (n4) at (4,4) {4};
  \foreach \from/\to in {n1/n2, n2/n3, n3/n4}
    \draw (\from) -- (\to);
\end{tikzpicture}
\end{center}
Then 
$\mathscr{L}(P)=\{1324, 1342, 3124,3142,3412\}$ and
$$ A_P(\zz)= z_1z_2z_{1'}z_{2'}z_{4'}+z_1z_2z_{1'}z_{2'}z_{3'} +z_1z_3z_{1'}z_{2'}z_{4'} +z_1z_2z_3z_{1'}z_{2'}+ z_1z_3z_{1'}z_{2'}z_{3'}.$$
\end{example}
Note that $A_P(\zz)$ is a polynomial in $2n$ variables, and homogeneous of degree $n+1.$ For anti-chains these polynomials were first considered by the first author in \cite{BrU}, where they were proven to be stable. 
An $(n-1)$-variable specialization for anti-chains was earlier defined in \cite{HV1}, but not proven to be stable.  

\begin{remark}\label{slots}
Let $\pi =\pi_1 \pi_2 \cdots \pi_n \in \sym_n$. For $0\leq i \leq n$, the $i$th \emph{slot} of $\pi$ is the ``space'' between $\pi_i$ and $\pi_{i+1}$, where $\pi_0=\pi_{n+1}=\infty$. A slot is uniquely determined  by an element of $B(\pi)$, namely $\pi_{i}'$ if  $\pi_i<\pi_{i+1}$ and $\pi_{i+1}$ if $\pi_i>\pi_{i+1}$. 

An \emph{internal slot} of $\pi$ is a slot which is not the first or the last slot of $\pi$.  An internal  slot is uniquely determined  by an element of $([n]\cup[n]') \setminus B(\pi)$, namely $\pi_{i+1}$ if  $\pi_i<\pi_{i+1}$ and $\pi_{i}'$ if $\pi_i>\pi_{i+1}$ ($1\leq i \leq n-1$). 
\end{remark}
We define the \emph{disjoint union} of two labeled posets $P=([n], \preceq_P)$  and $Q=([n], \preceq_Q)$ on ground sets $[n]$ and $[m]$ to be  the labeled poset $P \sqcup Q=([n+m], \preceq)$ whose set of relations is 
$$
\{ i \preceq j :  i, j \in [n] \mbox{ and } i \preceq_P j\} \cup 
\{ (n+i) \preceq (n+j) : i, j \in [m] \mbox{ and } i \preceq_Q j\}. 
$$
We want to see the effect on multivariate Eulerian polynomials upon taking disjoint unions. For two labeled posets $P$ and $Q$ with $A_P(\zz)$ and $A_Q(\zz)$ stable, we will analyze $A_{P\sqcup Q}(\zz)$ with respect to stability using free quasi-symmetric functions. 
In Section \ref{pps} we prove the following theorem.
\begin{theorem}
\label{PUQr}
Let $P$ and $Q$ be labeled posets. If $A_P(\zz)$ and $A_Q(\zz)$ are stable, then so is $A_{P\sqcup Q}(\zz)$. 
\end{theorem}
We also consider a more general definition of disjoint union. Let  $P=([n],\preceq_P)$ and $Q=([m],\preceq_Q)$  be labeled posets, and let $S\subseteq [n+m]$ be a set of size $n$. Order the element of $S$ and $T=[n+m]\setminus S$ in increasing order $s_1<\cdots < s_n$ and $t_1< \cdots < t_m$. Define $P\sqcup_S Q=([n+m], \preceq)$ to be the poset on $[n+m]$ with  relations 
$s_i \preceq s_j$ if and only if $i\preceq_P j$, and $t_i \preceq t_ j$ if and only if $i\preceq_Q j$. 
Similarly, if $P_1, \ldots , P_m$ are labeled posets with $|P_i|=n_i$, $i \in [m]$, and $\sum_{i=1}^m n_i=n$ we may define $P_1 \sqcup_{S_1}P_2 \sqcup_{S_2} \cdots \sqcup_{S_{m-1}}P_m$ for any ordered partition $S_1 \cup \cdots \cup S_m=[n]$ with $|S_i|=n_i$, for all $i \in [m]$. 

\begin{corollary}
\label{label}
Let $P=([n],\preceq_P)$ and $Q=([m],\preceq_Q)$ be two labeled posets. If $A_P(\zz)$ and $A_Q(\zz)$ are stable, then so is $A_{P\sqcup_S Q}(\zz)$ for any $S\subseteq [n+m]$ with $|S|=n$. Indeed $A_{P\sqcup_S Q}(\zz)$  and $A_{P\sqcup Q}(\zz)$ differ only by a permutation of the variables. 
\end{corollary}

\begin{proof}
Suppose $S \neq [n]$ and let $T=[n+m]\setminus S$. Then there are $s\in S$ and $t\in T$ such that  $t= s - 1$. Indeed let $s$ be the smallest element of $S$ such that there is a $t \in T$ such that $t < s$. Then $s-1 \in T$. 
Hence let $s\in S$ and $t\in T$ be such that $t= s - 1$. Consider $\tilde{S}=S\cup\{t\}\setminus \{s\}$. For $\pi \in \sym_{n+m}$, let $\tilde{\pi}$ be the permutation obtained by swapping the letters $s$ and $t$ in $\pi$. Clearly $\pi \in  \JJ(P\sqcup_S Q)$ if and only if $\tilde{\pi} \in \JJ(P\sqcup_{\tilde S} Q)$. Moreover, since $s$ and $t$ are not related,  if $s$ and $t$ are adjacent in $\pi$, then 
$$
\pi \in \JJ(P\sqcup_S Q) \Leftrightarrow \tilde{\pi} \in \JJ(P\sqcup_S Q) \Leftrightarrow \pi \in \JJ(P\sqcup_{\tilde{S}} Q)
\Leftrightarrow  \tilde{\pi} \in  \JJ(P\sqcup_{\tilde{S}} Q).
$$
If $s$ and $t$ are not adjacent in $\pi$, then $w_1(\tilde{\pi})$ is obtained from $w_1(\pi)$ by swapping the variables $z_{s}$ and $z_{t}$, as well as the variables $z_{s'}$ and $z_{t'}$. Hence $A_{P\sqcup_{\tilde{S}} Q}(\zz)$ is obtained from  $A_{P\sqcup_S Q}(\zz)$ by the same change of variables. 

If $\tilde{S}=[n]$ we are done. Otherwise continue the process with $S=\tilde{S}$. This process will terminate, and then $\tilde{S}=[n]$. Indeed the sum of the elements in $\tilde{S}$ is strictly smaller than the sum of the elements in $S$. 
\end{proof}

We argue that the natural setting for Theorem \ref{PUQr} is the algebra of free quasi-symmetric functions, see \cite{DHT, MR}.
Let $\FQS = \bigoplus_{n=0}^{\infty} \FQS_n $ be a $\CC-$linear vector space where $\FQS_n$ has basis $\{\pi: \pi\in \mathfrak{S}_n\}$. The (shuffle-) product on $\FQS$ may be defined on basis elements as 
 $$
 \pi \shuffle \sigma= \sum_{\tau \in \JJ(P_\pi \sqcup P_\sigma)} \tau, 
 $$
where $P_\pi$ is the labeled chain $\pi_1 \prec \cdots \prec \pi_n$. 
That is, the shuffle product $\pi \shuffle \sigma$ is the sum over all ways of interleaving the two permutations $\pi$ and $\hat{\sigma}$, where $\hat{\sigma} = (\sigma_1 + n) \cdots (\sigma_m +n)$. For example 
\begin{eqnarray*}
132 \shuffle 21 &=& 13254+ 13524+13542+15324+15342\\ &+&15432+ 51324+51342+51432+54132.
\end{eqnarray*}
Let $\EE=\{1,2,\ldots, 1', 2', \ldots \}.$ 
Extend $w_1$ linearly to a weight function $w_1 : \FQS \rightarrow \CC [z_e: e\in \EE]$.  We will now introduce two linear operators on 
$\CC [z_e: e\in \EE]$. 
For $e \in \EE,$ 
let $\eta_e$ be the linear \emph{creation operator} defined by first setting $z_e=0$ in a polynomial and then multiplying it by $z_e$. Moreover, for a finite set $S \subseteq \EE$, let 
$$
\quad \eta^S := \prod_{e \in S}\eta_e, 
$$
where $\eta^{\emptyset}$ is the identity operator. 
Let further $\partial^S$ be the \emph{annihilation operator}
$$
\partial^S := \prod_{e \in S}\frac \partial {\partial z_e},
$$
where $\partial^{\emptyset}$ is the identity operator. We also introduce an operation, $\Gamma_k$, that shifts the variables of a polynomial as 
$$
\Gamma_k(f(z_1, z_2, \ldots, z_{1'}, z_{2'}\ldots))= f(z_{1+k}, z_{2+k}, \ldots, z_{(1+k)'}, z_{(1+k)'}, \ldots).
$$
\begin{lemma}\label{shudif}
Let $f \in \FQS_n$ and $g \in \FQS_m$, where $mn\geq 1$. Let also 
$\FF= [n]\cup [n]'$ and 
$\GG=[n+1, n+m]\cup[n+1,n+m]'.$ Then $w_1(f \cdot g)$ only depends on $w_1(f)$ and $w_1(g).$ Moreover
$$
w_1(f \cdot g) = \Phi (w_1(f)\Gamma_n (w_1(g))),
$$
where 
$$
\Phi= \sum_{T,S}\eta^T\partial^S,
$$
and where the sum is over all $T \subseteq \FF, S \subseteq \GG$ for which $|S|=|T|+1 $. 
\end{lemma}

\begin{proof}
 It suffices to prove the lemma for basis elements. 
Let $\pi\in \mathfrak{S}_n$ and $ \sigma \in \sym_m$. 
 A permutation in $\mathscr{L}(P_{\pi} \sqcup P_{\sigma})$ is uniquely determined by a subset $S$ of the slots of $\sigma$ and a subset $T$  of the internal slots of $\pi$ for which $|S|=|T|+1$. Indeed we may factor $\pi$ as $\pi=v_1 \cdots v_k$ according to $T$, and $\hat{\sigma}:=\hat{\sigma}_0(\sigma_1+n) \cdots (\sigma_m+n) \hat{\sigma}_{m+1}$, where $\hat{\sigma}_0=\hat{\sigma}_{m+1}=\infty$, as  $\hat{\sigma}=w_1\cdots w_{k+1}$ according to $S$. Then we get a unique permutation $\tau \in \mathscr{L}(P_{\pi} \sqcup P_{\sigma})$ for which 
$\tau_0\tau_1\cdots \tau_{n+m+1}= w_1v_1w_2 \cdots v_kw_{k+1}$, where $\tau_0=\tau_{n+m+1}=\infty$. 

The next step is to see how the descent and ascent bottoms of $\pi$ and $\sigma$ are transferred to $\tau$. 

Consider the effect of inserting $v_i$ between $w_i$ and $w_{i+1}$. Let $s \in S$ be the element of $B=B(\sigma)$ that determines the slot between $w_i$ and $w_{i+1}$ (see Remark \ref{slots}). Since the letters of $v_i$ are smaller than those of $\hat{\sigma}$ the letter $s$ will be removed from $B$. This corresponds to the action of $\partial /\partial z_s$. 

Consider the effect of inserting $w_i$ between $v_{i-1}$ and $v_{i}$ for $1<i \leq k$.
Let $t \in T$ be  the element of $\FF \setminus B(\pi)$ that determines the internal slot between $v_{i-1}$ and $v_{i}$. Since the letters of $w_i$ are greater than those of $\pi$, the letter $t$ will be added to $B$. This corresponds to action of $\eta_t$. Nothing happens when $w_1$ or $w_{k+1}$ are inserted at the ends. The lemma now follows. 
\end{proof}

Since $w_1(fg)$ only depends on $w_1(f)$ and $w_1(g)$, Lemma \ref{shudif} provides a ``descent- ascent-bottom" algebra which is a quotient of $\FQS$. This algebra may thus be defined by 
$$
{\rm DAB}_n = \mathrm{span}_\CC\{ w_1(\sigma) : \sigma \in \sym_n\}, \mbox{ and } {\rm DAB}= \oplus_{n=0}^\infty {\rm DAB}_n, 
$$
with multiplication of homogeneous elements defined by 
$f\bullet 1=1\bullet f=f$ and 
$$
f\bullet g = \Phi(f\Gamma_n(g)),
$$
if $f \in {\rm DAB}_n$, $g\in {\rm DAB}_m$, where $mn \neq 0$. By Lemma \ref{shudif}, $w_1: \FQS \rightarrow {\rm DAB}$ is an algebra homomorphism. We will see in Section \ref{algebraDyck} that ${\rm DAB}$ may be viewed as an algebra of Dyck paths and that the dimension of ${\rm DAB}_n$ is the $n$th Catalan number $C_n= \binom {2n} n /(n+1)$. 

\section{An algebra of Dyck paths}
\label{algebraDyck}
Recall that a \emph{Dyck path} of length $2n$ is a path in $\NN \times \NN$ starting from $(0, 0)$ and ending in $(2n, 0)$ using $2n$ steps, where each step is represented by one of the vectors $(1, 1)$ and $(1, -1)$. We call $u=(1,1)$ an \emph{up step}, and $d=(1,-1)$ a \emph{down step}. The number of Dyck paths of length $n$ is equal to the $n$th Catalan number. For us it will be convenient to code a Dyck path $w_1, w_2, \ldots, w_{2n}$, where $w_i \in \{u,d\}$, as the word 
\begin{equation}\label{dycode}
v_1v_2\cdots v_{2n} = uw_1\cdots w_{2n-1}.
\end{equation} Since $w_{2n}$ is always a down-step we lose no information by this representation. 

Define operators $\partial_1,\partial_2, \ldots$ and $\eta_1, \eta_2, \ldots $ on the algebra of noncommutative polynomials, $\CC\langle u,d \rangle,$ as follows. If $v_1v_2 \cdots v_n$ is a word with letters in $\{u,d\}$ and $i$ is a positive integer, then  
$$
\partial_i(v_1\cdots v_n)= \begin{cases}
v_1\cdots v_{i-1} d v_{i+1} \cdots v_n &\mbox{ if } v_i=u, \\
0 &\mbox{ if } i>n \mbox{ or } v_i=d,
\end{cases}
$$
and dually 
$$
\eta_i(v_1\cdots v_n)= \begin{cases}
v_1\cdots v_{i-1} u v_{i+1} \cdots v_n &\mbox{ if } v_i=d, \\
0 &\mbox{ if } i>n \mbox{ or } v_i=u.
\end{cases}
$$
Moreover if $S \subseteq \{1,2,\ldots\}$ is a finite set, then $\partial^S= \prod_{i \in S}\partial_i$ and $\eta^S= \prod_{i \in S}\eta_i$. Endow $\CC\langle u,d \rangle$  with a product $\bullet$ given by $f \bullet 1= 1 \bullet f =f$ and 
$$
(w_1w_2 \cdots w_m) \bullet (v_1v_2 \cdots v_n) = \sum_{T,S} \eta^T(w_1 \cdots w_m) \partial^S (v_1 \cdots v_n), 
$$
where the sum is over all finite sets $S, T \subset \{1,2,\ldots \}$ such that $|S|=|T|+1$, whenever $mn >0$.  
\begin{lemma}\label{vw}
If $w=w_1w_2 \cdots w_{2m}, v=v_1v_2 \cdots v_{2n} \in \CC\langle u,d \rangle$ are Dyck paths represented as in 
\eqref{dycode}, then $w \bullet v$ is a sum of words which all represent Dyck paths.  
\end{lemma}

\begin{proof}
A word $w=w_1w_2 \cdots w_{2n}$ in the alphabet $\{u,d\}$ represents a Dyck path as in \eqref{dycode} if and only if $w_1=w_2=u$ and the corresponding path $w_1,\ldots, w_{2m}$ (in the $(x,y)$-plane) is a path from $(0,0)$ to $(2m,2)$ which crosses $y=1$ exactly once. Clearly the path corresponding to $\eta^T(w_1 \cdots w_{2m}) \partial^S (v_1 \cdots v_{2n})=P$ does not cross $y=1$ when $1<x \leq 2m$. Let $h(x)$ be the height of the path corresponding to $v_1 \cdots v_{2n}$ after $x$ steps. If $2n < 2n+x \leq 2(n+m)$ then the height in $P$ after $2n+x$ steps is at least $2|S| +2-2|T| +h(x)\geq h(x)$, since the path corresponding to $w_1 \cdots w_m$ ends at height $2$, and we have turned $|S|$ down steps to up steps and at most $|T|$ up steps to down steps. 
\end{proof}

By Lemma \ref{vw} we have a graded Dyck algebra 
$$
\mathcal{D}= \oplus_{n=0}^\infty \mathcal{D}_n,
$$
where $\mathcal{D}_n$ is the span of all Dyck paths $v_1v_2\cdots v_{2n}$ coded as in \eqref{dycode}.  We will now see that $\mathcal{D}$ is isomorphic to ${\rm DAB}$. First define an algebra homomorphism $\Theta : {\rm DAB} \rightarrow (\CC\langle u,d \rangle, \bullet)$ as follows. If $M_n$ is a monomial defining  a basis element of ${\rm DAB}_n$, let $\Theta(M_n)= v_1v_2\cdots v_{2n}$ be defined as follows. For each $i \in [n]$ let 
\begin{itemize}
\item $v_{2i-1}=u$ if and only if $x_i$ appears in $M_n$, and 
\item $v_{2i}=u$  if and only if $x_{i'}$ appears in $M_n$.
\end{itemize}
By construction of the product $\bullet$ on $\CC\langle u,d \rangle$, we see that $\Theta$ is an algebra homomorphism. 
\begin{theorem}
The map $\Theta$ is an algebra isomorphism between the algebras ${\rm DAB}$ and $\mathcal{D}$. 
\end{theorem}

\begin{proof}
Clearly $\Theta(z_1z_{1'})= uu$ is the unique basis element of $\mathcal{D}_1$. In $\FQS$ we have the identity 
$$
1^n = 1\shuffle 1 \shuffle \cdots \shuffle 1= \sum_{\sigma \in \sym_n} \sigma, 
$$ 
where $1 \in \sym_1$. Since $w_1 : \FQS \rightarrow {\rm DAB}$ and $\Theta : {\rm DAB} \rightarrow (\CC\langle u,d \rangle, \bullet)$ are  algebra homomorphisms, with $w_1(1)=z_1z_{1'}$, we see that 
$\Theta({\rm DAB}_n)$ is the span of all words in the support of 
$$
(uu)^n=(uu)\bullet \cdots \bullet (uu) \in \mathcal{D}_n. 
$$
Hence $\Theta : {\rm DAB} \rightarrow \mathcal{D}$ by Lemma \ref{vw}. The homomorphism $\Theta : {\rm DAB} \rightarrow \mathcal{D}$ is injective since $\Theta : {\rm DAB} \rightarrow (\CC\langle u,d \rangle, \bullet)$ is injective. To prove surjectivity it remains to prove that for any Dyck path $v \in \mathcal{D}_n$ there is a Dyck path $w \in \mathcal{D}_{n-1}$ such that $v$ is in the support of 
$(uu) \bullet w$. Let $w$ be the word obtained by first changing the first down step (say at position $j$) in $v$ to an up set and then  deleting the first two letters. Then $w$ is a Dyck path and $v=uu\partial_{j-2}(w)$ is in the support of   $(uu) \bullet w$. 
\end{proof}

\begin{example}
The product of $uudu, uuud \in \mathcal{D}_2$ is  $uudu \bullet uuud = uuduuudd + uuuududd + uuuuuddd+uududuud+uuuuddud+ uuduudud.$
\end{example}

\begin{remark}
There is a much studied graded algebra on rooted planar binary trees called 
the Loday-Ronco algebra \cite{LR}. The Loday-Ronco algebra is a sub-algebra of $\FQ$, and rooted planar binary trees are in bijection with Dyck paths. Hence it is natural to ask if this algebra and $\mathcal{D}$ are isomorphic. We have not found such an isomorphism. 
\end{remark}

\section{Products preserving stability}
\label{pps}
To prove that $\Phi$ in Theorem \ref{shudif} preserves stability we need two theorems on stable polynomials. The first theorem is a version of the celebrated Grace-Walsh-Szeg\H{o} theorem, see e.g. \cite[Proposition 3.4]{LYPSI}. 

Let $\Omega \subset \CC^n$. A polynomial $P(\zz) \in \CC[z_1, \ldots, z_n]$ is \emph{$\Omega$-stable} if 
$$
\zz \in \Omega \quad \mbox{ implies } \quad P(\zz) \neq 0.
$$

\begin{theorem}[]\label{gws}
Let $P(z_1,z_2, \ldots , z_n)$ be a polynomial, let $H \subset \CC$ be an open half-plane, and $1\leq k \leq n$. If $P$ is of degree at most one in $z_i$ for each $1\leq i \leq k$ and symmetric in the variables $z_1, \ldots , z_k$,  then $P(z_1,z_2, \ldots , z_n)$ is $H^n$-stable if and only if the polynomial 
$P(z_1,z_1, \ldots , z_1, z_{k+1}, \ldots,  z_n)$ is $H^{n-k+1}$-stable.
\end{theorem}

The next theorem is a special case of a recent characterization of stability preservers in \cite{LYPSI}. 
Let $\mathbb{C}_{\mathbf{1}}[\zz]=\mathbb{C}_{\mathbf{1}}[z_1, \ldots ,z_n]$ be the space of polynomials of degree at most one in $z_i$ for all $i$. For a linear operator $T:\mathbb{C}_{\mathbf{1}}[\zz] \rightarrow \mathbb{C}_{\mathbf{1}}[\zz]$ define its \emph{symbol} by
$$G_T(\zz,\ww)=T[(z_1+w_1)\cdots (z_n+w_n)]=\sum_{S\subseteq [n]}T(\zz^S)\ww^{[n]\backslash S}.$$ 

\begin{theorem}[\cite{LYPSI}]
\label{z+w}
Let $\Omega=\{z \in \CC : \Im(z)>0\}^n$ or $\Omega=\{z \in \CC : \Re(z)>0\}^n$ and let  $T:\mathbb{C}_{\mathbf{1}}[\zz] \rightarrow \mathbb{C}_{\mathbf{1}}[\zz]$ be a linear operator. 
\begin{itemize}
\item If $G_T(\zz,\ww)$ is $\Omega\times \Omega$-stable, then $T$ preserves $\Omega$-stability. 
\item If the rank of $T$ is greater than one and $T$ preserves $\Omega$-stability, then $G_T(\zz,\ww)$ is $\Omega \times \Omega$-stable.
\end{itemize}
\end{theorem}

\begin{lemma}\label{Malo}
Let $m \geq 2$ and $n\geq 1$ be integers. All zeros of the polynomial 
$$
\sum_{k=0}^n \binom n k \binom m {k+1} x^k
$$
are real and negative. 
\end{lemma}

\begin{proof}
The lemma is a consequence of Malo's theorem (see e.g. \cite[Theorem 2.4]{CC3}) which asserts that if $f=\sum_{k\geq 0} a_k x^k$ is a real-rooted polynomial and $g=\sum_{k\geq 0} b_k x^k$ is a real-rooted polynomial whose zeros all have the same sign, then the polynomial
$$
f \ast g = \sum_{k\geq 0} a_kb_k x^k 
$$
is real-rooted. Indeed, the polynomial in the statement of the theorem is $$x^{-1} (x(x+1)^n \ast (x+1)^m).$$ 
\end{proof}

We are now ready to state and prove the main theorem of this section. Note that Theorem \ref{PUQr} immediately follows from the theorem below.

\begin{theorem}
\label{fqs}
Let $f,g \in \FQS $ be two homogeneous elements  and let $w_1$ be defined as before.
If $w_1(f)$ and $w_1(g)$ are stable, then so is $w_1(f\cdot g)$.
\end{theorem}

\begin{proof}
Note that a homogeneous polynomial is $H^n$-stable for an open half-plane $H$ with boundary containing the origin if and only if it is $J^n$-stable for some (and then each) open half-plane $J$ with boundary containing the origin. Hence it suffices to prove that $\Phi$ preserves stability with respect to the open right half-plane (Hurwitz stability). Recall that $\Phi$ acts on multi-affine polynomials in the variables $\{z_e : e \in \FF \cup \GG\}$ where $\FF$ and $\GG$ are disjoint sets.   Now 
\begin{align*}
\partial^S \prod_{e \in \GG}(z_e+w_e) &= \prod_{e \in \GG \setminus S}(z_e+w_e)= \prod_{e \in \GG}(z_e+w_e)\prod_{e \in S}(z_e+w_e)^{-1}\quad \mbox{ and } \\
\eta^T \prod_{f \in \FF}(z_f+w_f) &= \prod_{f \in T}\!z_fw_f \!\!\!\prod_{f \in \FF \setminus T}\!\!(z_f+w_f)= \prod_{f \in \FF} (z_f+w_f) \prod_{f \in T}z_f w_f(z_f+w_f)^{-1}.
\end{align*}
Hence we may write the symbol of $\Phi$ as 
$$
G_\Phi(\zz,\ww) = \prod_{e \in \GG}(z_e+w_e)\prod_{f \in \FF}(z_f+w_f)   \sum_{S,T}   \prod_{e \in S} y_e \prod_{f \in T} x_f,
$$
where the sum is over all $S \subseteq \GG, T \subseteq \FF$ for which $|S|=|T|+1$, 
$$x_f=z_f w_f(z_f+w_f)^{-1}= \left( \frac{1}{w_f} + \frac{1}{z_f} \right)^{-1},$$ and  $y_e=(z_e+w_e)^{-1}$ for all $f \in \FF$ and $e\in \GG$. Since the open right half-plane is invariant under $z \mapsto z^{-1}$ it suffices to prove that the polynomial 
$$\sum_{S,T}   \prod_{e \in S} y_e \prod_{f \in T} x_f$$
is Hurwitz stable. This polynomial is symmetric in $\xx$ and in $\yy$, so by Theorem \ref{z+w} it remains to prove that the polynomial 
$$
y\sum_{k=0}^n \binom n k \binom m {k+1} x^ky^{k}=:yP(xy)
$$
is Hurwitz stable. The polynomial $P(x)$ has only real and negative zeros by Lemma~\ref{Malo}, so that $P(xy)$ is a product of factors of the form $a + xy$ where $a>0$. The product of two numbers in the open right half-plane is never a negative real number, from which the proof follows. 
\end{proof}

We will now generalize Theorem \ref{fqs} to other weights. 
Define weight functions $w_j : \FQ \rightarrow \CC[x, y,z_1,z_2,\ldots]$ for $2\leq j \leq 4$ by 
\begin{align*}
w_2(\pi) &= y^{|\pi|-\des(\pi)} \prod_{i \in \DB(\pi)} \!\!\!z_i,   \\
w_3(\pi) &= y^{\des(\pi) +1 } \prod_{i \in \AB(\pi)}\!\!\! z_i \, \, \text{and}  \\
w_4(\pi) &= x^{\des(\pi)+1}y^{|\pi|-\des(\pi)}.
\end{align*} 

For a finite set $A$ of indices, let $S_A$ be the symmetrization with respect to the variables indexed by $A$, that is,  
$$
S_A=\frac{1}{|A|!} \sum_{\pi \in \mathfrak{S}(A)} \pi,
$$ 
where $\pi$ acts on the variables of the polynomial by permuting them according to $\pi$.
For a proof of the next lemma see e.g. Theorem 1.2 and Proposition 1.5 in \cite{LYPSII}. 
\begin{lemma}\label{symi}
Let $A \subseteq [n]$ and suppose $f \in \CC[z_1,\ldots, z_n]$ is stable and has degree at most one in $z_i$ for each $i \in A$. Then $S_A(f)$ is stable. 
\end{lemma}

\begin{lemma}\label{goin}
Let $\{z_i \}_{i \in F}$ and $\{z_i \}_{i \in G}$ be disjoint set of variables and suppose that $f$ and $g$ are multi-affine polynomials that depend only on the variables indexed by $F$ and $G$, respectively. Let 
$$
\Phi = \sum_{S,T}\eta^T\partial^S,
$$
where the sum is over all $S \subseteq F$, $T \subseteq G$ for which $|S|=|T|+1$. 

If $A \subseteq F$ and $B \subseteq G$, then 
$$
S_AS_B \Phi (fg) = \Phi(S_Af \cdot S_Bg).
$$
\end{lemma}

\begin{proof}
The lemma follows
since $\Phi$ acts symmetrically on the variables in $F$ and $G$, and 
$$
\pi(\partial^S g)= \partial^{\pi(S)} \pi(g) \ \ \mbox{ and } \ \ \pi(\eta^T f)= \eta^{\pi(T)} \pi(f), 
$$ 
where $\pi(S)= \{ \pi(s) : s \in S\}$.
\end{proof}

In Lemma \ref{shudif} we proved that the weight $w_1$ of a product of two elements in $\FQS$ only depend on the weights ($w_1$) of the two elements. In the next lemma we prove the same statement for the weights $w_2, w_3$ and $w_4$.

\begin{lemma}\label{fori}
Let $f,g \in \FQ$ and $1 \leq i \leq 4$. Then $w_i(f\cdot g)$ only depends on 
$w_i(f)$ and $w_i(g)$. 

Moreover if $w_i(f)$ and $w_i(g)$ are homogenous and stable, then so is $w_i(f \cdot g)$. 
\end{lemma}
\begin{proof}
We prove the lemma for $i=4$. The other cases follow similarly. 
Let $f \in \FQ_n$ and $g \in \FQ_m$. For  $1\leq i \leq j$ let  $[i,j]=\{i,i+1, \ldots, j\}$ and  $[i,j]'= \{i', \ldots, j'\} $. Note that  $w_4(f \cdot g)$ carries precisely the same information as $S_{[1,n+m]} S_{[1,n+m]'} w_1(f \cdot g)$. Now, by applying Lemma \ref{goin} 
\begin{align*}
& S_{[1,n+m]} S_{[1,n+m]'} w_1(f\cdot g) = S_{[1,n+m]} S_{[1,n+m]'} \Phi \big(w_1(f)  \Gamma_n( {w_1(g)}) \big) \\ 
= & S_{[1,n+m]} S_{[1,n+m]'} S_{[1,n]} S_{[1,n]'}S_{[n+1,n+m]} S_{[n+1,n+m]'}\Phi \big(w_1(f)\Gamma_n({w_1(g)})\big) \\
= & S_{[1,n+m]} S_{[1,n+m]'} \Phi \big( FG\big),
\end{align*} 
where  $F= S_{[1,n]} S_{[1,n]'}w_1(f)$ and $G= S_{[n+1,n+m]} S_{[n+1,n+m]'} \Gamma_n({w_1(g)})$. This proves the first statement. 

Now, $F$ is stable if and only if $w_4(f)$ is stable by Theorem \ref{gws}. 
 This completes the proof by Theorem \ref{fqs} and Lemma \ref{symi}. 
\end{proof}

\section{Applications}
For a labeled poset $P$ we define a \emph{descent bottom $P$-Eulerian polynomial} by 
$$
A^{\rm  DB}_P(\zz) = \sum_{\pi \in \JJ(P)} \prod_{e \in \DB(\pi)} \!\!\!z_e.
$$

\begin{corollary}\label{dbprod}
Let $P$ and $Q$ be labeled posets on $[n]$ and $[m]$, respectively. Suppose  $A^{\rm  DB}_P(\zz)$ and $A^{\rm  DB}_Q(\zz)$ are stable. If $S \subset [n+m]$ is an $n$-set, then the polynomial $A^{\rm  DB}_{P\sqcup_S Q}(\zz)$ is stable. 
\end{corollary}

\begin{proof}
As in the proof of Corollary \ref{label} it suffices to consider the case when $S=[n]$. Suppose $A^{\rm  DB}_P(\zz)$ and $A^{\rm  DB}_Q(\zz)$ are stable. Then so are the homogenized polynomials 
$$
\sum_{\pi \in \JJ(P)} y^{n-\des(\pi)}\prod_{e \in \DB(\pi)} z_e \ \  \mbox{ and } \ \ \sum_{\pi \in \JJ(Q)} y^{m-\des(\pi)}\prod_{e \in \DB(\pi)} z_e 
$$
by \cite[Theorem 4.5]{BBL}. The corollary now follows from Lemma \ref{fori}. 
\end{proof}
Brenti \cite{Brenti} conjectured that if the univariate $P$-Eulerian polynomials $A_P(x)$ and $A_Q(x)$ of two labeled posets $P$ and $Q$ are real-rooted, then so is $A_{P \sqcup Q}(x)$. The conjecture was proved by Wagner in \cite{Wa1,Wa2}. We may now deduce it as an immediate corollary of Lemma \ref{fori}. 
\begin{corollary}
If $P$ and $Q$ are labeled posets such that $A_P(x)$ and $A_Q(x)$ are real-rooted, then so is $A_{P \sqcup Q}(x)$.
\end{corollary}
\begin{proof}
The proof follows as the proof of Corollary \ref{dbprod}, using $w_4$ instead of $w_2$. 
\end{proof}

Stembridge \cite{Stem1,Stem2} studied the \emph{peak polynomial} associated to a labeled poset $P$. A peak in a permutation 
$\pi \in \sym_n$ is an index $1<i<n$ such that $\pi_{i-1} < \pi_i > \pi_{i+1}$. Let $\Lambda(\pi)$  be the set of peaks of $\pi$ and define 
$$
\bar{A}_P(x)= \sum_{\pi \in \JJ(P)} x^{|\Lambda(\pi)|}.
$$
Let us now define  a \emph{multivariate peak polynomial}.
For $\pi \in \sym_n$ let 
$$
N(\pi) = \{ \pi_i : \pi_{i-1} < \pi_i >\pi_{i+1}, i \in [n] \}\cup \{ \pi_i : \pi_{i-1} > \pi_i < \pi_{i+1}, i \in [n] \}, 
$$ 
where $\pi_0=\pi_{n+1}=\infty$, 
be the \emph{peak-valley set} of $\pi$ and define a multivariate peak polynomial by
$$
\bar{A}_P(\zz):= \sum_{\pi \in \JJ(P)} \prod_{e \in N(\pi)}  \! \! \!  z_e. 
$$
Note that $|N(\pi)|= 2|\Lambda(\pi)|+1$ so that 
$$
\bar{A}_P(x,x,\ldots)=x \bar{A}_P(x^2).
$$
Recall that a polynomial $P(\zz) \in \CC[\zz]$ is said to be \emph{Hurwitz stable} if $P(\zz) \ne 0$ for all $\zz \in \CC^n$ with $\Re(z_i)\geq 0,  $ for all $1\leq i \leq n$.
\begin{proposition}
Let $P$ be a labeled poset. If $A_P(\zz)$ is stable, then $\bar{A}_P(\zz)$ is Hurwitz stable and $\bar{A}_P(x)$ is real-rooted. 
\end{proposition}
\begin{proof}
If  $A_P(\zz)$ is stable, then $A_P(\zz)$ is Hurwitz stable by homogeneity. 
Set $z_{e'}=z_e$ for all $e \in E$ in $A_P(\zz)$ and denote the resulting polynomial by $H(\zz)$. Then 
$$
\bar{A}_P(\zz) = \Psi\left(H(\zz)\right),
$$
where $\Psi$ is the linear operator that maps a monomial $\prod_{e \in \EE}z_e^{k_e}$ to $\prod_{e \in \EE}z_e^{k_e \! \! \! \mod 2}$, where $k_e \! \!  \mod 2=0$ if $k_e$ is even, and 
$k_e \! \!  \mod 2=1$ if $k_e$ is odd. The polynomial $\bar{A}_P(\zz)$ is Hurwitz stable since $\Psi$ preserves Hurwitz stability, see \cite[Proposition4.19]{COSW}.  Since $\bar{A}_P(\zz)$ is Hurwitz stable we may rotate the variables and deduce that 
$$
\bar{A}_P(ix,ix,\cdots)= ix\bar{A}_P(-x^2), \ \ \ \ (\mbox{where } i = \sqrt{-1}),
$$
is stable. Hence $\bar{A}_P(x)$ is real-rooted. 
\end{proof}

An natural question, which is not addressed in this paper, is whether $\bar{A}_{P \sqcup Q}$ is Hurwitz stable whenever $\bar{A}_P$ and $\bar{A}_Q$ are Hurwitz stable, for any two labeled posets $P$ and $Q$.

Let $P=([n],\preceq_P)$ and $Q([m],\preceq_Q)$ be two labeled posets. We define the \emph{ordinal sum} of $P$ and $Q$ to be the labeled poset $P\oplus Q=([n+m],\preceq)$ with the following set of relations:
\begin{align*}
\{ i \preceq j :  i, j \in [n] \mbox{ and } i \preceq_P j\} \cup \\
\{ (n+i) \preceq (n+j) : i, j \in [m] \mbox{ and } i \preceq_Q j\} \cup \\
\{ i \preceq (n+j): i\in [n], j \in [m] \}. 
\end{align*}
 
In Section \ref{disjoint} we saw the effect on multivariate Eulerian polynomials upon taking disjoint unions. Now we will study the effect for ordinal sums. 
For a labeled poset $P=([n], \preceq_P)$ define $P_0=([n+1],\preceq_{P_0})$ to be the poset where $1 \preceq_{P_0} j$ for $2\leq j \leq n+1$ and  $i\preceq_{P_0}j$ if $i-1 \preceq_P j-1$ for $i,j \in \{2,\ldots, n+1\}$.

\begin{lemma}
\label{oplus}
Let $P=([n],\preceq_P)$ be a poset such that $A_P(\zz)$ is stable and let $Q$ be a poset such that $A_{Q_0}(\zz)$ is stable. Then $A_{P \oplus Q}(\zz)$ is stable. 
\end{lemma}

\begin{proof}
Clearly 
$$\mathscr{L}(P \oplus Q)=\{ \pi_1 \cdots \pi_n (\sigma_{1}+n) \cdots (\sigma_{m}+n) : \pi_1 \cdots \pi_n\in \mathscr{L}(P),  \sigma_1 \cdots \sigma_m \in     \mathscr{L}(Q) \}  
$$ 
and $ \mathscr{L}(Q_0)= \{1      (\sigma_{1}+1) \cdots (\sigma_{m}+1)  : \sigma_1 \cdots \sigma_m \in     \mathscr{L}(Q)   \}$ from which 
$$A_{P \oplus Q}(\zz) = \frac{A_P(\zz) \Gamma_{n}(A_{Q_0}(\zz)) }{z_{n+1}z_{(n+1)'}}$$
follows. Thus $A_{P \oplus Q}(\zz)$ is stable. 
\end{proof}

 Define a \emph{naturally labeled decreasing tree}, $T$, recursively as follows.
 \begin{itemize}
 \item[1)] Either $T=T_0:=(\{1\},\emptyset)$, the antichain on one element, or
 \item[2)]$T=(T_1 \sqcup_{S_1}T_2 \sqcup_{S_2} \cdots \sqcup_{S_{m-1}}T_m) \oplus T_0$,  for some ordered partition $S_1 \cup \cdots \cup S_m=[n]$, where $T_i$ is a naturally labeled decreasing tree for all $i \in [m]$. 
 \end{itemize}
 That is, a naturally labeled decreasing tree is a labeled poset whose Hasse diagram is a decreasing tree with the root at the top.
 
 \begin{example}
 Let $T_1$ and $T_2$ be the naturally labeled decreasing trees below.
 
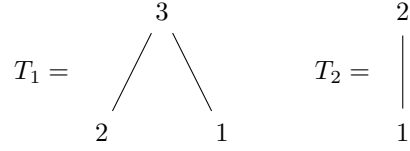
\begin{center}
 \begin{tikzpicture}
  [scale=.4,auto=left,every node/.style={circle}]
\node (n5) at (-2,2) {$T_1=$};
  \node (n1) at (0,0) {2};
  \node (n2) at (2,4)  {3};
  \node (n3) at (4,0)  {1};
  
  \node (n5) at (8,2) {$T_2=$};
   \node (n5) at (10,0) {1};
   \node (n6) at (10,4) {2};
  
  \foreach \from/\to in {n1/n2, n2/n3,  n5/n6}
    \draw (\from) -- (\to);
   
\end{tikzpicture}
 \end{center}
 Then
  
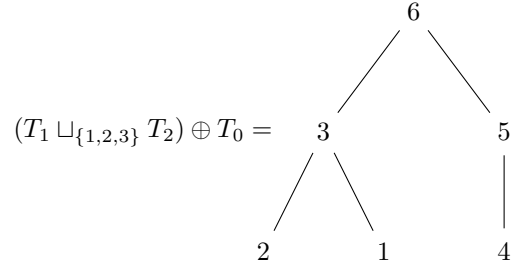
\begin{center}
 \begin{tikzpicture}
  [scale=.4,auto=left,every node/.style={circle}]
\node (n5) at (-4,4) {$(T_1\sqcup_{\{1,2,3\}}T_2) \oplus T_0=$};
  \node (n1) at (0,0) {2};
  \node (n2) at (2,4)  {3};
  \node (n3) at (4,0)  {1};

   \node (n5) at (8,0) {4};
   \node (n6) at (8,4) {5};
  
  \node (n7) at (5,8) {6};
  
  \foreach \from/\to in {n1/n2, n2/n3,  n5/n6, n2/n7, n6/n7}
    \draw (\from) -- (\to);
   
\end{tikzpicture}
 \end{center}
 
 \end{example}
 A \emph{naturally labeled decreasing forest} is a disjoint union 
 $$F=T_1 \sqcup_{S_1} T_2 \sqcup_{S_2} \cdots \sqcup_{S_{k-1}} T_k$$ of naturally labeled decreasing trees. 
\begin{corollary}
If $F$ is a naturally labeled decreasing forest, then $A_F(\zz)$ is stable.
\end{corollary}

\begin{proof}
The operations defining naturally labeled decreasing trees and forests preserve stability by Corollary \ref{label} and Lemma \ref{oplus}. Hence the corollary follows by induction on the size of $F$. 
\end{proof}
The \emph{dual} of a labeled poset $P=([n], \preceq)$ is the poset $P^*=([n], \preceq^{*})$,  where $a\preceq^* b$ if and only if $b \preceq a$.

\begin{proposition}
\label{dual}
If $P=([n],\preceq)$ is a poset such that $A_P(\zz)$ is stable, then $A_{P^*}(\zz)$ is stable. 
In fact, $A_{P^*}(z_1,\ldots z_n , z_{1'},\ldots , z_{n'})= A_P (z_{1'},\ldots z_{n'},z_1,\ldots , z_n).$ 
\end{proposition}

\begin{proof}
First note that $\pi^*=\pi_n \cdots \pi_1 \in \mathscr{L}(P^*)$ if and only if $\pi=\pi_1\cdots \pi_n \in \mathscr{L}(P).$ Hence $\DB (\pi)= \AB (\pi^*)$ and $\AB (\pi)= \DB (\pi^*)$, and the proposition follows. 
\end{proof}

\begin{corollary}
If $F$ is the dual of a naturally labeled decreasing forest, then $A_F(\zz)$ is stable.
\end{corollary}

\end{document}